\documentclass[titlepage]{amsart}

\usepackage{amsmath,amssymb,amsthm,mathrsfs}
\usepackage{graphicx}
\usepackage{amsaddr}

\theoremstyle{plain}
\newtheorem{theorem}{Theorem}[section]
\newtheorem{lemma}[theorem]{Lemma}

\newtheorem{corollary}[theorem]{Corollary}
\newtheorem{definition}[theorem]{Definition}
\newtheorem{remark}[theorem]{Remark}

\numberwithin{equation}{section}

\DeclareMathOperator{\oA}{A}

\DeclareMathOperator{\oB}{B}
\DeclareMathOperator{\oC}{C}
\DeclareMathOperator{\oD}{D}

\DeclareMathOperator{\oJ}{J}

\DeclareMathOperator{\oL}{L}

\DeclareMathOperator{\oM}{M}
\DeclareMathOperator{\oN}{N}
\DeclareMathOperator{\oP}{P}

\newcommand{\mP}{\mathbb{P}}
\newcommand{\mS}{\mathbb{S}}

\begin{document}

\title{Quadratic decomposition of bivariate orthogonal polynomials}

\author{Am\'ilcar Branquinho}
\address{CMUC, Department of Ma\-the\-ma\-tics, University of Coimbra, Apartado 3008, EC Santa Cruz, 3001-501 Coimbra, Portugal.}
\email[Corresponding author]{ajplb@mat.uc.pt}

\author{Ana Foulqui\'e Moreno}
\address{Departamento de Matem\'atica, Universidade de Aveiro, 3810-193 Aveiro, Portugal.}
\email{foulquie@ua.pt}

\author{Teresa E. P\'erez}
\address{Instituto 
de Matem\'aticas IMAG \&
Departamento de Matem\'{a}tica Aplicada, Facultad de Ciencias. Universidad de Granada (Spain).}
\email{tperez@ugr.es}




\begin{abstract}
We describe bivariate polynomial sequences orthogonal to a symmetric weight function in terms of several bivariate polynomial sequences orthogonal with respect to Christoffel transformations of the initial weight under a quadratic transformation.
We analyze the construction of a symmetric bivariate orthogonal polynomial sequence from a given one, orthogonal to a weight function defined on the positive plane. 
In this description plays an important role a sort of B\"acklund type matrix transformations for the involved three term matrix coefficients. 
We take as a case study relations between symmetric orthogonal polynomials defined on the ball and on the simplex.
\end{abstract}

\maketitle

\section{Motivation}

In \cite{Xu98, Xu01} the author found connections between even orthogonal polynomials on the ball and simplex polynomials in $d$ variables.
Following \cite[section 4]{Xu01}, let
\begin{align*}
W^\mathbf{B}(x_1,x_2,\ldots, x_d)= W(x_1^2,x_2^2, \ldots, x_d^2)
\end{align*}
be a weight function defined on the unit ball on $\mathbb{R}^d$, and let 
\begin{align*}
W^\mathbf{T}(u_1,u_2,\ldots, u_d) = \frac{1}{\sqrt
{u_1\,u_2\,\cdots \,u_d}} \, W(
{u_1},
{u_2},\ldots ,
{u_d} ), &&
 (u_1,u_2,\ldots, u_d)\in \mathbf{T}^d,
\end{align*}
where $\mathbf{T}^d$ is the unit simplex on $\mathbb{R}^d$. For $n\geqslant0$, and $\alpha\in\mathbb{N}^d_0$ a multi-index, let $S_{2n,\alpha}(x_1,x_2, \ldots,x_d)$ be an orthogonal polynomial associated to the weight function $W^\mathbf{B}$ of even degree in each of its variables. Then Y. Xu proved that it can be written in terms of orthogonal polynomials on the simplex as
\begin{align*}
S_{2n,\alpha}(x_1,x_2, \ldots, x_d)=P_{n,\alpha} (x^2_1,x^2_2, \ldots, x^2_d), &&
|\alpha|  
= n,
\end{align*}
where $P_{n,\alpha} (x^2_1,x^2_2, \ldots, x^2_d)$ is, for each $n \geqslant 0$, an orthogonal polynomial of total degree $n$ on the simplex associated to $W^\mathbf{T}$. 
In this way, there exists an important \emph{partial} connection between classical ball polynomials and simplex ones.

Inspired by these relations, we try to analyze the situation for the leftover polynomials in this procedure,~\emph{i.e.}, we want to know the properties of the polynomials with odd powers that were left in the above identification. We succeed doing so in a general framework, showing that these polynomials are related to new families of bivariate orthogonal polynomials, resulting from a Christoffel modification that we explicitly identify. Hence we have a totally answer in the case $d = 2$, that generalizes the one given by T. Chihara in \cite{Ch78}.

The paper is organized as follows.
In Section~\ref{sec2} we state the basic tools and results that we will need along the paper. The Section~\ref{symm-d2} is devoted to describe symmetric monic orthogonal polynomial sequences, starting with the basic properties and regarding how the polynomials are. In Section~\ref{symm_2} we analyze the quadratic decomposition process. This will be done in an equivalent procedure,~\emph{i.e.}, given a symmetric monic orthogonal polynomial sequence, we separate it in four families of polynomials in a \emph{zip} way, and we deduce the inherit properties of orthogonality for each of the four families, obtaining that they are Christoffel modifications of the quadratic transformation of the original weight function.
As a converse result, we construct a symmetric monic orthogonal polynomial sequence from a given one.

In Section \ref{Backlund_S} we give relations between the matrix coefficients of the three term relations of the involved families.
In addition, the matrix coefficients of the Christoffel transformations for the four families of orthogonal polynomials are given in terms of the matrix coefficients of the three term relations for the symmetric polynomials.
These matrices enable us to reinterpret the block Jacobi matrix associated with the four orthogonal polynomials sequences in terms of a $L \, U$ or $U L$ representation.

Finally, in the last Section, we complete the study started in \cite{Xu98, Xu01}  describing explicitly the four families of orthogonal polynomials on the simplex deduced from a symmetric polynomial system orthogonal on the ball.

\section{Basic Facts}\label{sec2}

For each $n\geqslant 0$, let $\Pi_n$ 
denote the linear space of bivariate polynomials of total degree not greater than $n$, and let
$\Pi=\bigcup_{n\geqslant 0} \Pi_n$. A polynomial $p(x,y)\in\Pi_n$ is a linear combination of monomials, \emph{i.e.}, 
\begin{align*}
p(x,y) = \sum_{m=0}^{n}\sum_{i=0}^m a_{m-i,i} x^{m-i}y^{i}, && a_{m-i,i}\in \mathbb{R},
\end{align*}
and it is of total degree $n$ if $\sum_{i=0}^n |a_{n-i,i}| >0$.

A polynomial $p(x,y)$ of total degree $n$ is called \emph{centrally symmetric}~if $p(-x,-y) = (-1)^n \, p(x,y).$ As a consequence, if the degree of the polynomial is even, then it only contains monomials of even degree, and if the degree of the polynomial is odd, then the polynomial only contains monomials of odd~powers.

We will say that a polynomial of partial degree $h$~in~$x$ is $x$-\emph{symmetric} if
\begin{align*}
p(-x,y) = (-1)^h \, p(x,y), && \forall (x,y).
\end{align*}
Therefore, if $h$ is even, $p(x,y)$ only contains even powers in $x$, and if $h$ is odd, it only contains odd powers in $x$. Analogously, we define the $y$-\emph{symmetric} for a polynomial of degree $k$ in $y$ as
\begin{align*}
p(x,-y) = (-1)^k \, p(x,y), && \forall (x,y).
\end{align*}
A $y$-\emph{symmetric} polynomial of degree $k$ in $y$ only contains odd powers in $y$ when $k$ is an odd number, and it contains only even power in $y$ when $k$ is even.

Obviously, if a polynomial is $x$-\emph{symmetric} and $y$-\emph{symmetric}, then it is \emph{centrally symmetric}. In this case, if the polynomial has total degree $n$, with degree $h$ in $x$ and degree~$k$ in~$y$, and $n=h+k$, then if $h$ is an odd number (respectively if $k$ is an odd number), then the polynomial only contains odd powers in $x$ (respectively, odd powers in $y$), and if $h$ is even (respectively, if $k$ is even), then it only contains even powers in $x$ (respectively, it contains only even powers in $y$).

For each $n\geqslant 0$, let $\mathbb{X}_n$ denote the $(n+1)\times 1$ column vector
$$
\mathbb{X}_n=
\begin{bmatrix}
x^n & x^{n-1} y & \cdots & xy^{n-1} & y^n
\end{bmatrix}^{\top},
$$
where, as usual, the superscript means the transpose. Then $\{\mathbb{X}_n\}_{n\geqslant 0}$ is called the \emph{canonical basis} of $\Pi$. As in~\cite{DX14}, for $n \geqslant 0$, we denote by
$$
\oL_{n,1}=
\left[ \begin{array}{@{}c|c@{}}
\begin{matrix}
1 \\
 & \ddots \\
 & & 1
\end{matrix} 
 & 
\begin{matrix}
0 \\ \vdots \\ 0
\end{matrix} 
\end{array} \right],
 \qquad
\oL_{n,2}=
\left[ \begin{array}{@{}c|c@{}}
\begin{matrix}
0 \\ \vdots \\ 0
\end{matrix} &
\begin{matrix}
1 \\
 & \ddots \\
 & & 1
\end{matrix} 
\end{array} \right],
$$
such~that \quad $ \oL_{n,1}\,\mathbb{X}_{n+1}=x\,\mathbb{X}_{n}$ \quad and \quad $\oL_{n,2}\,\mathbb{X}_{n+1}=y\,\mathbb{X}_{n}$.

For $i, j = 0,1$, and $n\geqslant0$, we introduce the matrices: $\oJ_{n}^{(i,j)}$ of dimension $(n+1)\times (2n+1+i+j)$ in the following way
\begin{align*}
\oJ_{n}^{(i,j)}
= \begin{bmatrix} j_{h,k}^{(n,i,j)} \end{bmatrix}_{h,k=0}^{n\times (2n-1+i+j)},
\end{align*}
such that $j_{h,2h+j}^{(n,i,j)}=1$, for $0\leqslant h\leqslant n$, the rest of the elements are zero. In particular,
\begin{align*}
\oJ_{n}^{(0,0)} & =
\begin{bmatrix}
1 & 0 & 0 &\cdots & 0 & 0\\ 
0 & 0 & 0 &\cdots & 0 & 0 \\ 
0 & 0 & 1 &\cdots & 0 & 0 \\ 
\vdots & \vdots & \vdots & \ddots & \vdots & \vdots\\
0 & 0 & 0 &\cdots & 0 & 1 
\end{bmatrix}, &
\oJ_{n}^{(1,0)} & = 
\begin{bmatrix}
1 & 0 & 0 &\cdots & 0 & 0\\ 
0 & 0 & 0 &\cdots & 0 & 0 \\ 
0 & 0 & 1 &\cdots & 0 & 0 \\ 
\vdots & \vdots & \vdots & \ddots & \vdots& \vdots\\
0 & 0 & 0 &\cdots & 1 & 0 
\end{bmatrix}, \\
\oJ_{n}^{(0,1)} & = \begin{bmatrix}
0 & 1 & 0 & 0 &\cdots & 0 & 0 \\ 
0 & 0 & 0 & 0 &\cdots & 0 & 0 \\ 
0 & 0 & 0 & 1 &\cdots & 0 & 0 \\ 
\vdots & \vdots & \vdots & \vdots & \ddots & \vdots & \vdots\\
0 & 0 & 0 & 0 &\cdots & 0 & 1 
\end{bmatrix}, &
\oJ_{n}^{(1,1)} & = \begin{bmatrix}
0 & 1 & 0 & 0 & \cdots & 0 & 0\\ 
0 & 0 & 0 & 1 &\cdots & 0 & 0 \\ 
0 & 0 & 0 & 0 &\cdots & 0 & 0 \\ 
\vdots & \vdots & \vdots & \vdots & \ddots & \vdots & \vdots\\
0 & 0 & 0 & 0 &\cdots & 1 & 0 
\end{bmatrix}.
\end{align*}
Observe that the $\oJ$-matrices are obtained from the identity matrices by introducing columns of zeros. The objective of these matrices is to extract the odd or the even elements in a vector of adequate~size. The transpose of these matrices introduce zeroes into a vector in the odd or even positions.

A simple computation allows us to prove the next result.

\begin{lemma}\label{J-L}
For $n\geqslant0$ and $k=1,2$, the following relations hold:
\begin{align*}
 \oJ_{n}^{(0,0)} \, \oL_{2n,k} = \oJ_{n}^{(2-k,k-1)}, &&
 \oJ_{n}^{(1,1)} \, \oL_{2n+2,k} = \oL_{n,k}\, \oJ_{n+1}^{(k-1,2-k)}, \\
  \oJ_{n}^{(k-1,2-k)} \, \oL_{2n+1,k} = \oJ_{n}^{(1,1)}, && 
 \oJ_{n}^{(2-k,k)} \, \oL_{2n+1,k} = \oL_{n,k}\, \oJ_{n+1}^{(0,0)}. 
\end{align*}
\end{lemma}

\subsection{Orthogonal polynomial systems (OPS)}

Let $\{ \oP_{n,m}(x,y): 0\leqslant m \leqslant n, n\geqslant 0\}$ denote a basis of $\Pi$ such that, for a 
fixed~$n\geqslant 0$,
$\deg \oP_{n,m}(x,y) = n$, and the set $\{\oP_{n,m}(x,y): 0\leqslant m \leqslant n\}$ contains $n+1$ 
linearly independent polynomials of total degree exactly~$n$.
We can write the vector of~polynomials
\begin{align*}
\mP _n= \begin{bmatrix}
\oP_{n,0}(x,y) & \oP_{n,1}(x,y) & \cdots & \oP_{n,n}(x,y)
\end{bmatrix}^{\top}.
\end{align*}
The sequence of polynomial vectors of increasing size
$\{\mP _n\}_{n\geqslant0}$ is called a \emph{polynomial system}~(PS), and it is a basis of $\Pi$. We say that is a \emph{monic polynomial sequence} if every entry has the form
\begin{align*}
P_{n,j} (x,y) = x^{n-j}\,y^j + \sum_{m=0}^{n-1}\sum_{i=0}^m a_{m-i,i} x^{m-i}y^{i}, && 0\leqslant j \leqslant n.
\end{align*}

Let $W(x,y)$ be a weight function defined on a domain $ \Omega \subset\mathbb{R}^2$, and we suppose the existence of every moment,
\begin{align*}
\mu_{h,k} = \int_{\Omega } x^h\,y^k \, W(x,y) \,\mathrm{d} x \mathrm{d} y < +\infty, &&
h, k \geqslant 0 .
\end{align*}
As usual, we define the inner product
\begin{align}\label{ip}
( p, q ) = \int_{ \Omega } p(x,y)\,q(x,y)\, W(x,y) \, \mathrm{d} x \mathrm{d} y, && p, q \in \Pi,
\end{align}
and remember how the inner product acts over polynomial matrices. Let $\oA=
\begin{bmatrix}
a_{i,j}(x,y)
\end{bmatrix}_{i,j=1}^{h,k}$
and
$\oB = \begin{bmatrix} b_{i,j}(x,y) \end{bmatrix}_{i,j=1}^{l,k}$
be two polynomial matrices. The action of \eqref{ip} over polynomial matrices is
defined as the $h \times l$ matrix (\emph{cf.}~\cite{DX14}),
\begin{align*}
( \oA, \oB ) = \int_{ \Omega } \oA(x,y)\, \oB(x,y)^\top\, W(x,y) \, \mathrm{d} x \mathrm{d} y =
\begin{bmatrix}
\displaystyle
\int_{ \Omega } c_{i,j}(x,y) W(x,y) \, \mathrm{d} x \mathrm{d} y
\end{bmatrix}_{i,j=1}^{h,l},
\end{align*}
where $\oC = \oA \cdot \oB^\top = \begin{bmatrix} c_{i,j}(x,y) \end{bmatrix}_{i,j=1}^{h,l}$.

A~PS $\{\mP _n\}_{n\geqslant 0}$ is an
\emph{orthogonal polynomial system (OPS) with respect to $(\cdot, \cdot)$}~if
\begin{align*}
( \mP _n, \mP_m )
= \begin{cases}
\mathtt{0}_{(n+1) \times (m+1)}, & n\ne m,\\
\mathbf{P}_n, & n=m,
\end{cases}
\end{align*}
where $\mathbf{P}_n$ is a positive-definite symmetric matrix of size $n+1$, and~$\mathtt{0}_{(n+1) \times (m+1)}$, or~$\mathtt{0}$ for short, is the zero matrix of adequate size. It was proved \cite{DX14} that there exists a unique monic orthogonal polynomial system associated to~$W(x,y)$, and we will call MOPS for short.

In this work we will use Christoffel modifications of a weight function given by a multiplication of a polynomial of degree~$1$.
In the next Lemma we recall the relations between the involved monic~OPS (\cite{APPR14}).

\begin{lemma}\label{shortRel}
Let $W(x,y)$ be a weight function defined on a domain $\Omega\subset\mathbb{R}^2$, and let $\lambda(x,y) = a\,x + b\,y$ be a polynomial with $|a| + |b| >0$, such that $W^\mathbf *(x,y) = \lambda(x,y)\,W(x,y)$ is again a weight function on $\Omega$. Let 
$\{\mathbb{P}_n\}_{n\geqslant0}$ and $\{\mathbb{P}^\mathbf *_n\}_{n\geqslant0}$ be the respective monic OPS.
Then, for all $n \geqslant 1$,
\begin{align*}
\mP_{n} & = \mP^\mathbf *_{n} + \oM_{n}\,\mP^\mathbf *_{n-1},\\
\lambda(x,y)\,\mP^\mathbf *_{n} & = \big( a \oL_{n,1} + b \oL_{n,2} \big) \,\mP_{n+1} 
+ \oN_{n}\,\mP_{n}, 
\end{align*}
where 
\begin{align*}
\oM_{n} = \mathbf{P}_n\,(a\,\oL_{n-1,1}^\top + b\,\oL_{n-1,2}^\top)\,(\mathbf{P}_{n-1}^\mathbf *)^{-1}, &&
\oN_{n} = \mathbf{P}_n^\mathbf *\,\mathbf{P}_{n}^{-1},
\end{align*}
and
\begin{align*}
\mathbf{P}_n =  \int_{\Omega} \mP_n\,\mP_n^\top\,W(x,y)\, \mathrm{d}x \mathrm{d}y,
 \qquad 
\mathbf{P}^\mathbf *_n =  \int_{\Omega} \mP^\mathbf *_n\,(\mP^\mathbf *_n)^\top\,W^\mathbf{*}(x,y)\, \mathrm{d}x \mathrm{d}y ,
\end{align*}
are non-singular matrices of size $(n+1)$.
\end{lemma}

\section{Symmetric Monic Orthogonal Polynomial Sequences}\label{symm-d2}

A weight function $W(x,y)$ defined on $\Omega \subset \mathbb{R}^2$ is called \emph{centrally symmetric} (\emph{cf.} \cite[p.~76]{DX14}) if satisfies
\begin{align*}
(x,y)\in \Omega \Rightarrow (-x,-y) \in \Omega && \text{and} && W(-x,-y) = W(x,y), && \forall (x,y)\in \Omega.
\end{align*}
Therefore, by a natural change of variables, we get
\begin{align*}
\mu_{h,k} =  \int_{\Omega} x^h y^k W(x,y)\,\mathrm{d}x \mathrm{d}y = \int_{\Omega} (-x)^h(-y)^kW(-x,-y)\,\mathrm{d}x \mathrm{d}y =  (-1)^{h+k} \mu_{h,k},
\end{align*}
and then $\mu_{h,k} =0$, for $h+k$ an odd integer number.

We introduce an additional definition of symmetry.

\begin{definition}
We say that a weight function $W(x,y)$ is $x$-\emph{symmetric} if
\begin{align*}
(x,y)\in \Omega \Rightarrow (-x,y) \in \Omega, && \text{and} && W(-x,y) = W(x,y), && \forall (x,y)\in \Omega.
\end{align*}
Analogously, the weight function is $y$-\emph{symmetric} if 
\begin{align*}
(x,y)\in \Omega \Rightarrow (x,-y) \in \Omega, && \text{and} && W(x,-y) = W(x,y),&& \forall (x,y)\in \Omega.
\end{align*}
A ~$x$-\emph{symmetric} and~$y$-\emph{symmetric} weight function is called~$x\,y$-\emph{sym\-met\-ric}.
\end{definition}

Obviously, if $W(x,y)$ is $x\,y$-\emph{symmetric} then it is centrally symmetric. As a consequence, if
$W(x,y)$ is $x\,y$-\emph{symmetric}, then $\mu_{h,k} =0$ when, at least one, $n$ or $m$ are odd numbers.

Let $\{\mS_n\}_{n\geqslant0}$ be the MOPS associated with a $x\,y$-\emph{symmetric} weight function satisfying
\begin{align}\label{IP_S}
(\mS_n,\mS_m) = \int_{\Omega} \mS_n(x,y)\,\mS_m(x,y)^\top\,W(x,y)\,\mathrm{d}x \mathrm{d}y = \begin{cases}
\mathtt{0} , & n\neq m,\\
\mathbf{S}_n , & n=m,
\end{cases}
\end{align}
where $\mathbf{S}_n$ is a $(n+1)$ positive-definite symmetric matrix.

\begin{lemma} \label{xy-sim}
If the explicit expression of every vector polynomial is given by
\begin{align*}
\mS_n(x,y)
= \begin{bmatrix}
S_{n,0}(x,y) & S_{n,1}(x,y) & S_{n,2}(x,y) & \cdots & S_{n,n}(x,y)
\end{bmatrix}^\top,
\end{align*}
where 
\begin{align*}
S_{n,k}(x,y) &= \sum_{i=0}^{\lfloor (n-k)/2\rfloor} \sum_{j=0}^{\lfloor k/2\rfloor} a_{i,j}^{n,k} \, x^{n-k-2i}\,y^{k-2j}, && 0\leqslant k \leqslant n ,
\end{align*}
with $a_{0,0}^{n,k} = 1$,
then the polynomials are $x\,y$-\emph{symmetric}, that is, 
\begin{align*}
S_{n,k}(x,y) = (-1)^n S_{n,k}(-x,-y) = (-1)^{n-k} S_{n,k}(-x,y) = (-1)^k S_{n,k}(x,-y).
\end{align*}
\end{lemma}

When a bivariate polynomial $S_{n,k}(x,y)$ is $x\,y$-\emph{symmetric} then it has the same parity order in every variable,~\emph{i.e.}, if the partial degree in the first variable $x$ is even (respectively,~odd), then all powers in $x$ are even (respectively,~odd), and analogously, if the partial degree in the second variable $y$ is even (respectively,~odd), then all powers in $y$ are even (respectively,~odd).

Therefore, the vector polynomial $\mS_{n}(x,y)$ can be separated in a \emph{zip} way, attending to the parity of the powers of $x$ and $y$ in its entries. In fact, for even, respectively odd degree, we get

\begin{align}\label{separation}
\mS_{2n} = 
\begin{bmatrix}
S_{2n,0} \\ 0 \\ S_{2n,2} \\ 0 \\ \vdots \\ 0 \\ S_{2n,2n}
\end{bmatrix} +
\begin{bmatrix}
0 \\ S_{2n,1} \\ 0 \\ S_{2n,3} \\ \vdots \\S_{2n,2n-1} \\ 0
\end{bmatrix},  \quad \mS_{2n+1} = 
\begin{bmatrix}
S_{2n+1,0} \\ 0 \\ S_{2n+1,2} \\ 0 \\ \vdots \\ S_{2n+1,2n}\\0\end{bmatrix} +
\begin{bmatrix}
0 \\ S_{2n+1,1} \\ 0 \\ S_{2n+1,3} \\ \vdots \\0\\S_{2n+1,2n+1}
\end{bmatrix}.
\end{align}

\begin{lemma} \label{representation}
We can express the monic orthogonal polynomial vectors as
\begin{align}
\label{S-P-R1}
\mS_{2n}(x,y) & = \mP_n^{(0,0)}(x^2, y^2) + x\,y\,\mP_{n-1}^{(1,1)}(x^2, y^2),\\
\label{S-P-R2}
\mS_{2n+1}(x,y) & = x\, \mP_n^{(1,0)}(x^2, y^2) + y\,\mP_{n}^{(0,1)}(x^2, y^2),
\end{align}
where, for $n\geqslant 0$,

\medskip

\noindent
$\mP_n^{(0,0)}(x^2, y^2)$ is a vector of size $(2n+1)\times 1$ whose odd entries are independent monic polynomials of exact degree~$n$ on $(x^2,\,y^2)$, and its even entries are zeroes,

\noindent
$\mP_{n}^{(1,1)}(x^2, y^2)$ is a vector of size $(2n+3)\times 1$ whose even entries are independent monic polynomials of exact degree~$n$ on $(x^2,\,y^2)$, and its odd entries are zeroes,

\noindent
$\mP_n^{(1,0)}(x^2, y^2)$ is a vector of size $(2n+2)\times 1$ whose odd entries are independent monic polynomials of exact degree~$n$ on $(x^2,\,y^2)$, and its even entries are zeroes,

\noindent 
$\mP_n^{(0,1)}(x^2, y^2)$ is a vector of size~$(2n+2)\times 1$ whose even entries are independent monic polynomials of exact degree~$n$ on $(x^2,\,y^2)$, and its odd entries are zeroes.

\end{lemma}

These families will be called \emph{big} vector polynomials associated with $\{\mS_n\}_{n\geqslant0}$.
We must observe that the \emph{big} families are formed by vectors of polynomials in the variables $(x^2, y^2)$, that contains polynomials of independent degree intercalated with zeros.

Our objective is to \emph{extract} the odd entries in the vectors $\mP^{(i,0)}_{n}(x, y)$, and the even entries in $\mP^{(i,1)}_{n}(x, y)$, for $i=0,1$. 

\begin{lemma}
For $n\geqslant0$, and $i, j = 0,1$, we define the $(n+1)\times 1$ vector of polynomials
\begin{align*}
\widehat{\mP}_n^{(i,j)}(x,y) = \oJ_{n}^{(i,j)}\mP_n^{(i,j)}(x,y), && n\geqslant0.
\end{align*}
Then, its entries are independent polynomials of exact degree $n$, and therefore, the sequences of vectors of polynomials $\{\widehat{\mP}_n^{(i,j)}\}_{n\geqslant0}$ are \emph{polynomial systems}.
\end{lemma}

\section{Quadratic decomposition process}\label{symm_2}

Taking into account Lemma \ref{xy-sim}, we start studying the inherit properties of orthogonality of the polynomial systems $\{\widehat{\mP}_n^{(i,j)}\}_{n\geqslant0}$, for $i,j=0,1$.

\begin{theorem}\label{direct}
Let $\{\mS_n\}_{n\geqslant0}$ be a $x\,y$-symmetric monic orthogonal polynomial system associated with a weight function $W(x,y)$ defined on a domain $\Omega\subset \mathbb{R}^2$.
Then, the four families of polynomials $\{\widehat{\mP}_n^{(i,j)}\}_{n\geqslant0}$, for $i,j=0,1$, 
defined in terms of the big ones by~\eqref{S-P-R1},~\eqref{S-P-R2}
and
$\widehat{\mP}_n^{(i,j)} = \oJ_{n}^{(i,j)}\,\mP_n^{(i,j)}$, are monic orthogonal polynomial systems (MOPS) associated respectively, with the weight~functions 
\begin{align*}
W^{(0,0)}(x,y) =& \dfrac{1}{4}\,\dfrac{1}{\sqrt{x\,y}}\,W(\sqrt{x},\sqrt{y}) , 
 \\
W^{(1,0)}(x,y) =& \dfrac{1}{4}\,\sqrt{\dfrac{x}{y}}\,W(\sqrt{x},\sqrt{y})=x\, W^{(0,0)}(x,y) ,
 \\
W^{(0,1)}(x,y) =& \dfrac{1}{4}\,\sqrt{\dfrac{y}{x}}\,W(\sqrt{x},\sqrt{y})=y\,W^{(0,0)}(x,y) ,
 \\
W^{(1,1)}(x,y) =& \dfrac{1}{4}\,\sqrt{x\,y}\,W(\sqrt{x},\sqrt{y})=x\,y\,W^{(0,0)}(x,y) ,
\end{align*}
for all $ (x,y) \in \Omega^\mathbf * = \{(x,y)\in\mathbb{R}^2: x, y \geqslant 0, (\sqrt{x}, \sqrt{y})\in \Omega\}.$
\end{theorem}

\begin{figure}
 \centering
\includegraphics[width=0.475\textwidth]{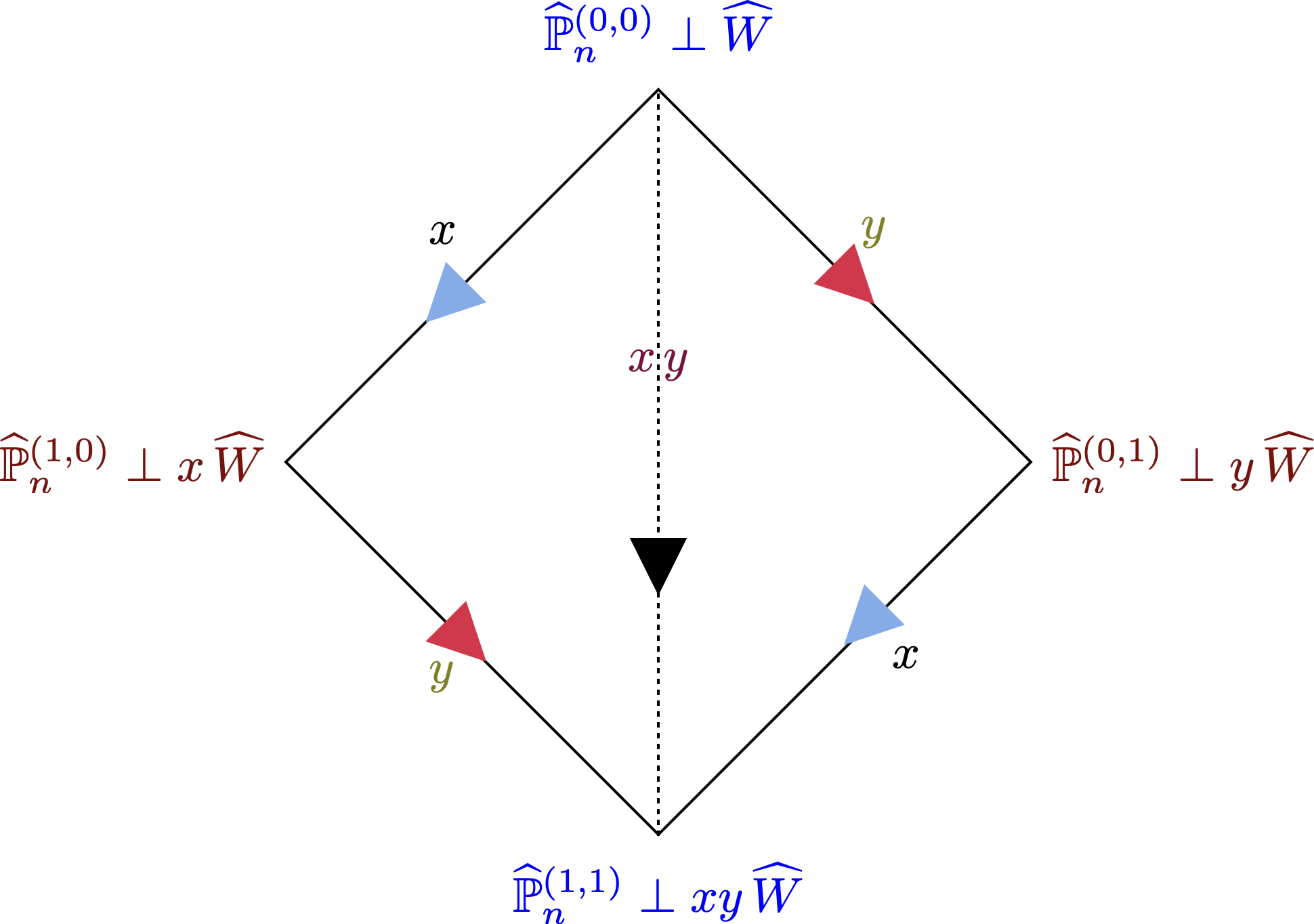}
\caption{Relation between the four weight functions and the corresponding polynomial systems.}
\end{figure}

\begin{proof}
From expression \eqref{S-P-R1} and the $x\,y$-symmetry of the inner product~\eqref{IP_S}, we get 
\begin{align*}
(\mS_{2n}(x,y), \mS_{2m} (x,y))  = & (\mP_n^{(0,0)}(x^2, y^2), \mP_m^{(0,0)}(x^2, y^2)) \\
& + (x\,y\,\mP_{n-1}^{(1,1)}(x^2, y^2), x\,y\,\mP_{m-1}^{(1,1)}(x^2, y^2)),
\end{align*}
On the one hand, if $n\neq m$, then $(\mS_{2n}(x,y),\mS_{2m}(x,y)) = \mathtt{0}$ if and only if
\begin{align*}
(\mP_n^{(0,0)}(x^2, y^2), \mP_m^{(0,0)}(x^2, y^2)) = \mathtt{0}, && 
(x\,y\,\mP_{n-1}^{(1,1)}(x^2, y^2), x\,y\,\mP_{m-1}^{(1,1)}(x^2, y^2)) = \mathtt{0},
\end{align*}
because the positivity of the inner product.

On the other hand, if $n = m$, then $(\mS_{2n}(x,y),\mS_{2n}(x,y)) = \mathbf{S}_{2n}$, a symmetric positive-definite matrix, and defining the matrices 
\begin{align*} 
\mathbf{P}^{(0,0)}_n = & (\mP_n^{(0,0)}(x^2, y^2), \mP_n^{(0,0)}(x^2, y^2)) 
=  \int_{\Omega} \mP^{(0,0)}_n(x^2,y^2)\mP^{(0,0)}_n(x^2,y^2)^\top W(x,y)\mathrm{d}x \mathrm{d}y \\
\mathbf{P}^{(1,1)}_{n-1} = & (x\,y\,\mP_{n-1}^{(1,1)}(x^2, y^2), x\,y\,\mP_{n-1}^{(1,1)}(x^2, y^2)) \\
= & \int_{\Omega} \mP^{(1,1)}_{n-1}(x^2,y^2)\mP^{(1,1)}_{n-1}(x^2,y^2)^\top x^2 y^2 W(x,y)\mathrm{d}x \mathrm{d}y, 
\end{align*}
they are symmetric of size $(2n+1)\times (2n+1)$, since $W(x,y)$ is a weight function on $\Omega$, and $x^2\,y^2\,W(x,y)$ is a positive definite Christoffel perturbation. Therefore, 
\begin{align*}
\mathbf{S}_{2n} = \mathbf{P}^{(0,0)}_n + \mathbf{P}^{(1,1)}_{n-1}.
\end{align*}
In order to recover a MOPS, we need to do a change of variable, and multiply times a suitable $\oJ$-matrix to shrink the vectors to an adequate size.
Hence, we define the change of variable $u= x^2$, $v=y^2$, and the integration domain will be defined~by
$\Omega^\mathbf * = \{(u,v)\in\mathbb{R}^2: u, v \geqslant 0, (\sqrt{u}, \sqrt{v})\in \Omega\}.$

Then, the PS $\{\widehat{\mP}_{n}^{(0,0)}\}_{n\geqslant 0} = \{\oJ_{n}^{(0,0)}\,\mP_{n}^{(0,0)}(u,v)\}_{n\geqslant 0}$ is orthogonal in the form
\begin{align*}
(\widehat{\mP}_{n}^{(0,0)},&\widehat{\mP}_{n}^{(0,0)})^{(0,0)}
=  \dfrac{1}{4}\int_{\Omega^\mathbf *} \widehat{\mP}_{n}^{(0,0)}(u,v)\widehat{\mP}_{n}^{(0,0)}(u,v)^\top W^{(0,0)}(u,v)\mathrm{d} u \mathrm{d} v\\
= & \dfrac{1}{4} \oJ_{n}^{(0,0)} \int_{\Omega^\mathbf *} 
\mP_{n}^{(0,0)}(u,v)\mP_{n}^{(0,0)}(u,v)^\top \dfrac{1}{\sqrt{u\,v}}\,W(\sqrt{u},\sqrt{v})\mathrm{d} u \mathrm{d} v (\oJ_{n}^{(0,0)})^\top \\
= & \oJ_{n}^{(0,0)} \int_{\Omega} 
\mP_{n}^{(0,0)}(x^2,y^2)\mP_{n}^{(0,0)}(x^2,y^2)^\top\,W(x, y) \, \mathrm{d} x \mathrm{d} y (\oJ_{n}^{(0,0)})^\top \\
= & \oJ_{n}^{(0,0)}\,\mathbf{P}_n^{(0,0)}\, (\oJ_{n}^{(0,0)})^\top = \widehat{\mathbf{P}}_n^{(0,0)},\\
(\widehat{\mP}_{n}^{(0,0)},&\widehat{\mP}_{m}^{(0,0)})^{(0,0)}
= \mathtt{0}.
\end{align*}
Moreover, $\widehat{\mathbf{P}}_n^{(0,0)}$ is a symmetric $(n+1)$ full rank matrix since $\{\widehat{\mP}_{n}^{(0,0)}\}_{n\geqslant 0}$ is a PS.

Acting in the same way on $\{\widehat{\mP}_{n}^{(1,1)}\}_{n\geqslant 0} = \{\oJ_{n}^{(1,1)}\,\mP_{n}^{(1,1)}(u,v)\}_{n\geqslant 0}$, we can prove the orthogonality relations
\begin{align*}
(\widehat{\mP}_{n}^{(1,1)},&\widehat{\mP}_{n}^{(1,1)})^{(1,1)}
= \dfrac{1}{4}\int_{\Omega^\mathbf *} \widehat{\mP}_{n}^{(1,1)}(u,v)\widehat{\mP}_{n}^{(1,1)}(u,v)^\top
\,W^{(1,1)}(u,v)\mathrm{d} u \mathrm{d} v\\
= & \dfrac{1}{4} \oJ_{n}^{(1,1)} \int_{\Omega^\mathbf *} 
\mP_{n}^{(1,1)}(u,v)\mP_{n}^{(1,1)}(u,v)^\top\,\sqrt{u\,v}\,W(\sqrt{u},\sqrt{v})\mathrm{d} u \mathrm{d} v 
(\oJ_{n}^{(1,1)})^\top \\
= & \oJ_{n}^{(1,1)}\,\int_{\Omega} \mP_{n}^{(1,1)}(x^2,y^2)\,\mP_{n}^{(1,1)}(x^2,y^2)^\top\,
x^2\,y^2\,W(x,y) \mathrm{d}x \mathrm{d}y\,(\oJ_{n}^{(1,1)})^\top\\
= & \oJ_{n}^{(1,1)}\,\mathbf{P}_{n}^{(1,1)}\, (\oJ_{n}^{(1,1)})^\top = \widehat{\mathbf{P}}_{n}^{(1,1)},\\
(\widehat{\mP}_{n}^{(1,1)}, &\widehat{\mP}_{m}^{(1,1)})^{(1,1)} = \mathtt{0}.
\end{align*}
Now, we multiply two odd symmetric polynomials, use \eqref{S-P-R2} and the $x\,y$-symmetry, obtaining
\begin{align*}
 (\mS_{2n+1}(x,y),\mS_{2m+1}(x,y)) = & (x\, \mP_n^{(1,0)}(x^2, y^2), x\, \mP_m^{(1,0)}(x^2, y^2)) \\
 &+ (y\,\mP_{n}^{(0,1)}(x^2, y^2), y\,\mP_{m}^{(0,1)}(x^2, y^2)).
\end{align*}
Using the same reasoning as in the even case, and defining the PS 
$\{\widehat{\mP}_{n}^{(1-j,j)}\}_{n\geqslant 0}$ $= \{\oJ_{n}^{(1-j,j)}\,\mP_{n}^{(1-j,j)}(u,v)\}_{n\geqslant 0} $, for $j=0,1$, we prove that they are orthogonal, as
\begin{align*}
(\widehat{\mP}_{n}^{(1-j,j)},\widehat{\mP}_{n}^{(1-j,j)})^{(1-j,j)}
= & \dfrac{1}{4}\int_{\Omega^\mathbf *} \widehat{\mP}_{n}^{(1-j,j)}(u,v)\widehat{\mP}_{n}^{(1-j,j)}(u,v)^\top W^{(1-j,j)}(u,v)
\mathrm{d} u \mathrm{d} v\\
= & \oJ_{n}^{(1-j,j)} \,\mathbf{P}_n^{(1-j,j)}\, (\oJ_{n}^{(1-j,j)})^\top = \widehat{\mathbf{P}}_{n}^{(1-j,j)},\\
(\widehat{\mP}_{n}^{(1-j,j)},\widehat{\mP}_{m}^{(1-j,j)})^{(1-j,j)}
= &\mathtt{0},
\end{align*}
which ends the proof.
\end{proof}

In a similar way we can prove the converse result.

\begin{theorem}\label{converse}
Let $\widehat{W}(x,y)$ be a weight function defined on $\Omega^\mathbf *\subset \mathbb{R}^2_+ = \{(x,y)\in\mathbb{R}^2: x, y \geqslant0\}$, and let $\{\widehat{\mP}^{(0,0)}_n\}_{n\geqslant0}$ be the corresponding monic OPS.
Let $\{\widehat{\mP}^{(1,0)}_n\}_{n\geqslant0}$, $\{\widehat{\mP}^{(0,1)}_n\}_{n\geqslant0}$, and $\{\widehat{\mP}^{(1,1)}_n\}_{n\geqslant0}$ be the respective MOPS associated with the modifications of the weight~function
\begin{align*}
W^{(1,0)}(x,y) = x\,\widehat{W}(x,y), &&
 W^{(0,1)}(x,y) = y\,\widehat{W}(x,y), && W^{(1,1)}(x,y) = x\,y\,\widehat{W}(x,y).
\end{align*}
Define the family of vector polynomials $\{\mathbb{S}_n\}_{n\geqslant0}$ by means of~\eqref{S-P-R1} and~\eqref{S-P-R2}, where $\{\mP_n^{(i,j)} = (\oJ_{n}^{(i,j)})^\top \, \widehat{\mP}^{(i,j)}_n\}_{n\geqslant0}$, for $i,j=0,1$.
Then, $\{\mathbb{S}_n\}_{n\geqslant0}$ is a $x\,y$-\emph{symmetric} monic orthogonal polynomial system associated with the weight function 
\begin{align*}
 W(x,y) = 4\,|x|\,|y|\,\widehat{W}(x^2, y^2) ,
 &&
(x,y) \in \Omega = \{(x,y)\in \mathbb{R}^2: (x^2, y^2) \in \Omega^\mathbf *\}.
\end{align*}
\end{theorem}

As a consequence of Theorem~\ref{converse}, and since $W^{(1,0)}(x,y)$, $W^{(0,1)}(x,y)$, and $W^{(1,1)}(x,y)$ are Christoffel modifications of the original weight function, following~\cite{APPR14} and Lemma~\ref{shortRel} there exist matrices of adequate size such that there exist short relations between that families of orthogonal polynomials. In the next section we will describe explicitly those relations.

\section{B\"acklund-type relations}\label{Backlund_S}

Orthogonal polynomials in two variables satisfy a three term relation in each variable (\emph{cf.}~\cite{DX14}) written in a vector form and matrix coefficients. In this section we want to relate the matrix coefficients of the three term relations for the monic orthogonal polynomial sequences involved in Theorems~\ref{direct} and~\ref{converse}.

If $\{\mathbb{S}_n\}_{n\geqslant0}$ is a MOPS associated with a centrally symmetric weight function, the three term relation takes a simple form. In fact,~\cite[Theorem~3.3.10]{DX14} states that a measure is symmetric if, and only if, it satisfies the three term relations
\begin{align}\label{S_monic_3tr}
\begin{cases}
 x \, \mS _n(x,y) = \oL_{n,1} \, \mS _{n+1}(x,y) + \Gamma_{n,1}\,\mS_{n-1}(x,y), \\
 y \, \mS _n(x,y) = \oL_{n,2} \, \mS _{n+1}(x,y) + \Gamma_{n,2}\,\mS_{n-1}(x,y),
\end{cases}
\end{align}
for $n\geqslant0$, where $\mS _{-1}(x,y)=0$, $\Gamma_{-1,k}=0$, and
\begin{align*}
 \Gamma_{n,k} = \mathbf{S}_n\, \oL_{n-1,k}^{\top}\, \mathbf{S}_{n-1}^{-1}, &&
n\geqslant 1, && k=1,2,
\end{align*}
are matrices of size $(n+1)\times n$ with $\mathrm{rank}\, \Gamma_{n,k} = n$, for $k=1,2$.

The four systems of monic orthogonal polynomials $\{\widehat{\mathbb{P}}^{(i,j)}_n\}_{n\geqslant0}$, with $i,j=0,1$, involved in Theorems~\ref{direct} and~\ref{converse}, satisfy the three term relations
\begin{align*}
\begin{cases}
x \, \widehat{\mP}^{(i,j)} _n(x,y) = \oL_{n,1} \, \widehat{\mP}^{(i,j)} _{n+1}(x,y)
+ \widehat{\oD}_{n,1}^{(i,j)}\,\widehat{\mP}^{(i,j)} _n(x,y) 
+ \widehat{\oC}_{n,1}^{(i,j)}\,\widehat{\mP}_{n-1}^{(i,j)}(x,y), \\
y \, \widehat{\mP}^{(i,j)} _n(x,y) = \oL_{n,2} \, \widehat{\mP}^{(i,j)} _{n+1}(x,y) 
+ \widehat{\oD}_{n,2}^{(i,j)}\,\widehat{\mP}^{(i,j)} _n(x,y) 
+ \widehat{\oC}_{n,2}^{(i,j)}\,\widehat{\mP}_{n-1}^{(i,j)}(x,y),
\end{cases}
\end{align*}
where $\widehat{\mP}^{(i,j)} _{-1}=0$, $\widehat{\oC}_{-1,k}^{(i,j)}=0$, $\widehat{\oD}_{n,k}^{(i,j)}$ and 
$\widehat{\oC}_{n,k}^{(i,j)}$ are matrices of respective sizes $(n+1)\times (n+1)$ and $(n+1)\times n$, such that
\begin{align*}
 \widehat{\oD}_{n,1}^{(i,j)}\, \widehat{\mathbf{P}}_{n}^{(i,j)} & = ( x\,\widehat{\mP}^{(i,j)} _n,\,
\widehat{\mP}^{(i,j)} _{n})^{(i,j)}, &&
\widehat{\oD}_{n,2}^{(i,j)}\, \widehat{\mathbf{P}}_{n}^{(i,j)} = ( y\,\widehat{\mP}^{(i,j)} _n,\,
\widehat{\mP}^{(i,j)} _{n})^{(i,j)}, \\
\widehat{\oC}_{n,1}^{(i,j)}\, \widehat{\mathbf{P}}_{n-1}^{(i,j)} & = \widehat{\mathbf{P}}_{n}^{(i,j)}\, \oL_{n-1,1}^{\top}, 
 &&
\widehat{\oC}_{n,2}^{(i,j)}\, \widehat{\mathbf{P}}_{n-1}^{(i,j)} = \widehat{\mathbf{P}}_{n}^{(i,j)}\, \oL_{n-1,2}^{\top} ,
\end{align*}
where $\widehat{\mathbf{P}}_{n}^{(i,j)} = (\widehat{\mP}^{(i,j)} _n,\,\widehat{\mP}^{(i,j)} _{n})^{(i,j)}$.
In addition, the $(n+1)\times n$ matrices $\widehat{\oC}_{n,k}^{(i,j)}$ have full rank $n$, for $i,j=0,1$ and $k=1,2$.

Suppose that the $x\,y$-\emph{symmetric} monic polynomial system $\{\mathbb{S}_n\}_{n\geqslant0}$ and the four families of MOPS are related by~\eqref{S-P-R1} and \eqref{S-P-R2}, where $\{\mP_n^{(i,j)} = (\oJ_{n}^{(i,j)})^\top \, \widehat{\mP}^{(i,j)}_n\}_{n\geqslant0}$, for $i,j=0,1$, are the respective families of \emph{big}~polynomials.

\begin{theorem}[B\"acklund-type relations]\label{Backlund}
In the above conditions, the following relations~hold, for all $n\geqslant 0$ and $k=1,2$,
\begin{align*}
\widehat{\oD}_{n,k}^{(0,0)} & = \oJ_{n}^{(0,0)}\,[\oL_{2n,k}\,\Gamma_{2n+1,k} + \Gamma_{2n,k}\,\oL_{2n-1,k}]\,(\oJ_{n}^{(0,0)})^\top, \\
\widehat{\oC}_{n,k}^{(0,0)} & = \oJ_{n}^{(0,0)}\,\Gamma_{2n,k}\,\Gamma_{2n-1,k}\,(\oJ_{n-1}^{(0,0)})^\top,\\
\widehat{\oD}_{n,k}^{(1,1)} & = \oJ_{n}^{(1,1)}\,[\oL_{2n+2,k}\,\Gamma_{2n+3,k} + \Gamma_{2n+2,k}\,\oL_{2n+1,k}]\,(\oJ_{n}^{(1,1)})^\top, \\
\widehat{\oC}_{n,k}^{(1,1)} & = \oJ_{n}^{(1,1)}\,\Gamma_{2n+2,k}\,\Gamma_{2n+1,k}\,(\oJ_{n-1}^{(1,1)})^\top,\\ 
\widehat{\oD}_{n,k}^{(1,0)} & = \oJ_{n}^{(1,0)}\,[\oL_{2n+1,k}\,\Gamma_{2n+2,k} + \Gamma_{2n+1,k}\,\oL_{2n,k}]\,(\oJ_{n}^{(1,0)})^\top, \\
\widehat{\oC}_{n,k}^{(1,0)} & = \oJ_{n}^{(1,0)}\,\Gamma_{2n+1,k}\,\Gamma_{2n,k}\,(\oJ_{n-1}^{(1,0)})^\top,\\
\widehat{\oD}_{n,k}^{(0,1)} & = \oJ_{n}^{(0,1)}\,[\oL_{2n+1,k}\,\Gamma_{2n+2,k} + \Gamma_{2n+1,k}\,\oL_{2n,k}]\,(\oJ_{n}^{(0,1)})^\top, \\
\widehat{\oC}_{n,k}^{(0,1)} & = \oJ_{n}^{(0,1)}\,\Gamma_{2n+1,k}\,\Gamma_{2n,k}\,(\oJ_{n-1}^{(0,1)})^\top.
\end{align*}
with the convention that the matrix with negative indices is taken as a zero matrix.
\end{theorem}

\begin{remark}
For $i=0,1$, we must observe that left multiplication by $\oJ_{n}^{(i,0)}$ eliminates the even rows of the matrices, and the left multiplication by $\oJ_{n}^{(i,1)}$ eliminates the odd rows of the matrices. The right multiplication by $(\oJ_{n}^{(i,0)})^\top$ eliminates the even columns, and the multiplication by $(\oJ_{n}^{(i,1)})^\top$ eliminates the odd columns of the matrices.
\end{remark}

We divide the proof in several lemmas starting from a useful one for symmetric~polynomials.

\begin{lemma}\label{simp}
Let $p_{i,j}(x,y), q_{i,j}(x,y)$, $i=0,1$, be polynomials of the same parity order. If
\begin{align*}
\begin{bmatrix}
p_{0,0}(x,y) & p_{0,1}(x,y)\\
p_{1,0}(x,y) & p_{1,1}(x,y)
\end{bmatrix}
\begin{bmatrix}
1\\
x
\end{bmatrix}
= \begin{bmatrix}
q_{0,0}(x,y) & q_{0,1}(x,y)\\
q_{1,0}(x,y) & q_{1,1}(x,y)
\end{bmatrix}
\begin{bmatrix}
1 \\
x
\end{bmatrix} ,
\end{align*}
then
\begin{align*}
\begin{bmatrix}
p_{0,0}(x,y) & p_{0,1}(x,y)\\
p_{1,0}(x,y) & p_{1,1}(x,y)
\end{bmatrix}
 = \begin{bmatrix}
q_{0,0}(x,y) & q_{0,1}(x,y)\\
q_{1,0}(x,y) & q_{1,1}(x,y)
\end{bmatrix}.
\end{align*}
\end{lemma}

Secondly, we deduce the relations between the \emph{big} families of polynomials.

\begin{lemma}\label{ShortRel_BigPoly}
The four \emph{big} families of polynomials $\{\mP^{(i,j)}_n\}_{n\geqslant0}$, for $i,j=0,1$, defined by~\eqref{S-P-R1} and~\eqref{S-P-R2}, are related by the expressions:
\begin{align}
\mP^{(0,0)}_{n}(x, y) & = \oL_{2n,k}\,\mP^{(2-k,k-1)}_{n}(x, y) + \Gamma_{2n,k}\,\mP^{(2-k,k-1)}_{n-1}(x, y)	,\label{eq-2}\\
x_k\,\mP^{(1,1)}_{n-1}(x, y) & = \oL_{2n,k}\,\mP^{(k-1,2-k)}_{n}(x, y) + \Gamma_{2n,k}\,\mP^{(k-1,2-k)}_{n-1}(x, y),\label{eq-1}\\
\mP^{(k-1,2-k)}_{n}(x, y) & = \oL_{2n+1,k}\,\mP^{(1,1)}_{n}(x, y) + \Gamma_{2n+1,k}\,\mP^{(1,1)}_{n-1}(x, y), \label{eq-4} \\
x_k\,\mP^{(2-k,k-1)}_{n}(x, y) & = \oL_{2n+1,k}\,\mP^{(0,0)}_{n+1}(x, y) + \Gamma_{2n+1,k}\,\mP^{(0,0)}_{n}(x, y),\label{eq-3}
\end{align}
for $k=1,2$ and denoting $x_1=x, x_2 = y$ for brevity.
\end{lemma}

\begin{proof}
The expressions~\eqref{S-P-R1} and~\eqref{S-P-R2} can be matrically rewritten in the following form
\begin{align}
\begin{bmatrix}
\mS_{2n}(x,y) \\
\mS_{2n+1}(x,y)
\end{bmatrix}
& = 
\begin{bmatrix}
\mP^{(0,0)}_{n}(x^2,y^2) & y\,\mP^{(1,1)}_{n-1}(x^2,y^2) \\
y\,\mP^{(0,1)}_{n}(x^2,y^2) & \mP^{(1,0)}_{n}(x^2,y^2) 
\end{bmatrix}\!
\begin{bmatrix}
1\\x
\end{bmatrix}
\label{S-P-Q-1}\\ 
& = 
\begin{bmatrix}
\mP^{(0,0)}_{n}(x^2,y^2) & x\,\mP^{(1,1)}_{n-1}(x^2,y^2) \\
x\,\mP^{(1,0)}_{n}(x^2,y^2) & \mP^{(0,1)}_{n}(x^2,y^2) 
\end{bmatrix}\!
\begin{bmatrix}
1\\y
\end{bmatrix}.
\label{S-P-Q-2}
\end{align}
We can write the first three term relations~\eqref{S_monic_3tr} in the form
$$
x \!
\begin{bmatrix}
\mS_{2n} \\
\mS_{2n+1}
\end{bmatrix}
  =
\begin{bmatrix}
\mathtt{0} & \mathtt{0} \\
\oL_{2n+1,1} & \mathtt{0}
\end{bmatrix} \!
\begin{bmatrix}
\mS_{2n+2} \\
\mS_{2n+3} 
\end{bmatrix}
+
\begin{bmatrix}
\mathtt{0} & \oL_{2n,1} \\
\Gamma_{2n+1,1} & \mathtt{0}
\end{bmatrix}\!
\begin{bmatrix}
\mS_{2n} \\
\mS_{2n+1} 
\end{bmatrix}
+
\begin{bmatrix}
\mathtt{0} & \Gamma_{2n,1} \\
\mathtt{0} & \mathtt{0}
\end{bmatrix}\!
\begin{bmatrix}
\mS_{2n-2} \\
\mS_{2n-1} 
\end{bmatrix}
$$
where we have omitted the arguments $(x,y)$ for simplicity. Substituting~\eqref{S-P-Q-1}, we get
\begin{multline*}
x  \!
\begin{bmatrix}
\mP^{(0,0)}_{n} & y\,\mP^{(1,1)}_{n-1}\\
y\,\mP^{(0,1)}_{n} & \mP^{(1,0)}_{n}
\end{bmatrix}\!
\begin{bmatrix}
1\\x
\end{bmatrix}
= \left\{
\begin{bmatrix}
\mathtt{0}^{\phantom{(0)}} & \mathtt{0} \\
\oL_{2n+1,1}^{\phantom{(0)}} & \mathtt{0}
\end{bmatrix}\!
\begin{bmatrix}
\mP^{(0,0)}_{n+1} & y\,\mP^{(1,1)}_{n} \\
y\,\mP^{(0,1)}_{n+1} & \mP^{(1,0)}_{n+1} 
\end{bmatrix} \right. \\
\left. +
\begin{bmatrix}
\mathtt{0} & \oL_{2n,1}^{\phantom{(0)}} \\
\Gamma_{2n+1,1}^{\phantom{(0)}} & \mathtt{0}
\end{bmatrix}\!
\begin{bmatrix}
\mP^{(0,0)}_{n} & y\,\mP^{(1,1)}_{n-1} \\
y\,\mP^{(0,1)}_{n} & \mP^{(1,0)}_{n} 
\end{bmatrix} + 
\begin{bmatrix}
\mathtt{0} & \Gamma_{2n,1}^{\phantom{(0)}} \\
\mathtt{0} & \mathtt{0}^{\phantom{(0)}}
\end{bmatrix}\!
\begin{bmatrix}
\mP^{(0,0)}_{n-1} & y\,\mP^{(1,1)}_{n-2} \\
y\,\mP^{(0,1)}_{n-1} & \mP^{(1,0)}_{n-1} 
\end{bmatrix}
\right\}\!
\begin{bmatrix}
1\\x
\end{bmatrix} .
\end{multline*}
where we have omitted the arguments $(x^2, y^2)$ of the \textit{big} polynomials for brevity. Now, since
\begin{align*}
x \!\begin{bmatrix}
1\\ x
\end{bmatrix}
 = \begin{bmatrix}
0 & 1 \\ x^2 & 0 
\end{bmatrix}\!
\begin{bmatrix}
1\\x
\end{bmatrix},
\end{align*}
and applying Lemma~\ref{simp}, we deduce
\begin{multline*}
 \begin{bmatrix}
\mP^{(0,0)}_{n} & y\,\mP^{(1,1)}_{n-1}\\
y\,\mP^{(0,1)}_{n} & \mP^{(1,0)}_{n}
\end{bmatrix}\!
\begin{bmatrix}
0 & 1 \\ x^2 & 0 
\end{bmatrix}
= \begin{bmatrix}
\mathtt{0}^{\phantom{(0)}} & \mathtt{0} \\
\oL_{2n+1,1}^{\phantom{(0)}} & \mathtt{0}
\end{bmatrix} \!
\begin{bmatrix}
\mP^{(0,0)}_{n+1} & y\,\mP^{(1,1)}_{n} \\
y\,\mP^{(0,1)}_{n+1} & \mP^{(1,0)}_{n+1} 
\end{bmatrix} \\
+ \begin{bmatrix}
\mathtt{0} & \oL_{2n,1}^{\phantom{(0)}} \\
\Gamma_{2n+1,1}^{\phantom{(0)}} & \mathtt{0}
\end{bmatrix}\! 
\begin{bmatrix}
\mP^{(0,0)}_{n} & y\,\mP^{(1,1)}_{n-1} \\
y\,\mP^{(0,1)}_{n} & \mP^{(1,0)}_{n} 
\end{bmatrix}
 + 
\begin{bmatrix}
\mathtt{0} & \Gamma_{2n,1}^{\phantom{(0)}} \\
\mathtt{0} & \mathtt{0}^{\phantom{(0)}}
\end{bmatrix}\! 
\begin{bmatrix}
\mP^{(0,0)}_{n-1} & y\,\mP^{(1,1)}_{n-2} \\
y\,\mP^{(0,1)}_{n-1} & \mP^{(1,0)}_{n-1} 
\end{bmatrix} .
\end{multline*}
We finally arrive to,
\begin{multline*}
\begin{bmatrix}
x^2\,y\,\mP^{(1,1)}_{n-1} & \mP^{(0,0)}_{n}\\
x^2\,\mP^{(1,0)}_{n} & y\,\mP^{(0,1)}_{n}
\end{bmatrix} 
= \begin{bmatrix}
\mathtt{0} & \mathtt{0} \\
\oL_{2n+1,1}\,\mP^{(0,0)}_{n+1} & \oL_{2n+1,1}\,y\,\mP^{(1,1)}_{n} 
\end{bmatrix} \\
+ \begin{bmatrix}
\oL_{2n,1}\,y\,\mP^{(0,1)}_{n} & \oL_{2n,1}\,\mP^{(1,0)}_{n} \\
\Gamma_{2n+1,1}\,\mP^{(0,0)}_{n} & \Gamma_{2n+1,1}\,y\,\mP^{(1,1)}_{n-1} 
\end{bmatrix}\
 + 
\begin{bmatrix}
\Gamma_{2n,1}\,y\,\mP^{(0,1)}_{n-1} & \Gamma_{2n,1}\,\mP^{(1,0)}_{n-1} \\
\mathtt{0} & \mathtt{0}
\end{bmatrix}. 
\end{multline*}
Then, after a convenient simplification and by introducing the variable $(x, y)$, we deduce the expressions~\eqref{eq-2},~\eqref{eq-4},~\eqref{eq-1}, and~\eqref{eq-3} for $k=1$. The same discussion can be done for the second variable using \eqref{S-P-Q-2}, taking $k=2$. 
\end{proof}

The identities in Lemma~\ref{ShortRel_BigPoly} can be used to deduce \emph{three terms} relations for the \emph{big} polynomial families. Apparently,~\eqref{Big-P0-TTR}-\eqref{Big-R1-TTR} are three term relations for the bivariate polynomials $\{\mP_n^{(i,j)}\}_{n\geqslant0}$, $i,j=0,1$, but the these \emph{big} families are not polynomial~systems.

\begin{lemma}\label{TTR_BigPoly}
The families of \emph{big} bivariate polynomials $\{\mP_n^{(i,j)}\}_{n\geqslant0}$, for $i,j=0,1$, satisfy the relations
\begin{align}
x_k\,\mP^{(0,0)}_{n}  = & \oL_{2n,k} \oL_{2n+1,k}\mP^{(0,0)}_{n+1}
+ [\oL_{2n,k} \Gamma_{2n+1,k} + \Gamma_{2n,k}\oL_{2n-1,k}]\mP^{(0,0)}_{n} \label{Big-P0-TTR}\\
&+ \Gamma_{2n,k} \Gamma_{2n-1,k}\mP^{(0,0)}_{n-1}, \nonumber\\
x_k\,\mP^{(1,1)}_{n-1}   = & \oL_{2n,k} \oL_{2n+1,k}\mP^{(1,1)}_{n}
 + [ \oL_{2n,k} \Gamma_{2n+1,k} + \Gamma_{2n,k} \oL_{2n-1,k}]\mP^{(1,1)}_{n-1}\label{Big-R0-TTR}\\ 
&+ \Gamma_{2n,k}\Gamma_{2n-1,k}\mP^{(1,1)}_{n-2},\nonumber\\
x_k\,\mP^{(1,0)}_{n} = & \oL_{2n+1,k} \oL_{2n+2,k}\mP^{(1,0)}_{n+1} 
+ [\oL_{2n+1,k} \Gamma_{2n+2,k} + \Gamma_{2n+1,k} \oL_{2n,k}]\mP^{(1,0)}_{n}\label{Big-P1-TTR}\\ 
& + \Gamma_{2n+1,k} \Gamma_{2n,k}\mP^{(1,0)}_{n-1},\nonumber\\
x_k\,\mP^{(0,1)}_{n} = & \oL_{2n+1,k} \oL_{2n+2,k}\mP^{(0,1)}_{n+1}
 + [ \oL_{2n+1,k} \Gamma_{2n+2,k} + \Gamma_{2n+1,k} \oL_{2n,k}]\mP^{(0,1)}_{n}\label{Big-R1-TTR} \\
& + \Gamma_{2n+1,k} \Gamma_{2n,k}\mP^{(0,1)}_{n-1},\nonumber
\end{align}
for $k=1,2$ and $x_1=x, x_2 = y$.
\end{lemma}

\begin{proof}
For $k=1,2$, relations are obtained multiplying~\eqref{eq-2} by $x_k$ and using~\eqref{eq-3}; substituting~\eqref{eq-2} in~\eqref{eq-3}; replacing~\eqref{eq-4} in~\eqref{eq-1}; and multiplying~\eqref{eq-4} by~$x_k$ and substituting~\eqref{eq-1}.
\end{proof}

From the three terms relations of the \emph{big} polynomials obtained in Lemma~\ref{TTR_BigPoly}, we can deduce the three term relations for the \emph{small} ones by a multiplications of an adequate~$\oJ$-matrix.
In fact, multiplying, respectively,~\eqref{Big-P0-TTR} by $\oJ_{n}^{(0,0)}$,~\eqref{Big-R0-TTR} by $\oJ_{n}^{(1,1)}$,~\eqref{Big-P1-TTR} by $\oJ_{n}^{(1,0)}$, and~\eqref{Big-R1-TTR} by $\oJ_{n}^{(0,1)}$, and making use of Lemma~\ref{J-L} we arrive to the following result.

\begin{lemma}
The families of \emph{small} bivariate polynomials $\{\widehat{\mP}_n^{(i,j)}\}_{n\geqslant0}$, for $i,j=0,1$, satisfy the three term relations
\begin{align}
x_k\,\widehat{\mP}^{(0,0)}_{n} = & \oL_{n,k}\widehat{\mP}^{(0,0)}_{n+1}
+ \oJ_n^{(0,0)} [ \oL_{2n,k} \Gamma_{2n+1,k} + \Gamma_{2n,k}\oL_{2n-1,k} ](\oJ_n^{(0,0)})^\top\widehat{\mP}^{(0,0)}_{n}\nonumber \\ 
 & + \oJ_n^{(0,0)}\Gamma_{2n,1} \Gamma_{2n-1,1}(\oJ_n^{(0,0)})^\top\widehat{\mP}^{(0,0)}_{n-1},\label{Small-P0-TTR}\\
x_k\,\widehat{\mP}^{(1,1)}_{n-1} = & \oL_{n-1,k} \widehat{\mP}^{(1,1)}_{n}
 + \oJ_{n-1}^{(1,1)}[ \oL_{2n,k} \Gamma_{2n+1,k} + \Gamma_{2n,k} \oL_{2n-1,k}](\oJ_{n-1}^{(1,1)})^\top
 \widehat{\mP}^{(1,1)}_{n-1}\nonumber\\
& + \oJ_{n-1}^{(1,1)}\Gamma_{2n,1}\Gamma_{2n-1,1}(\oJ_{n-2}^{(1,1)})^\top\widehat{\mP}^{(1,1)}_{n-2},\label{Small-R0-TTR}\\
x_k\,\widehat{\mP}^{(1,0)}_{n} = & \oL_{n,k}\widehat{\mP}^{(1,0)}_{n+1}
 + \oJ_n^{(1,0)} [\oL_{2n+1,k} \Gamma_{2n+2,k} + \Gamma_{2n+1,k} \oL_{2n,k}](\oJ_n^{(1,0)})^\top\widehat{\mP}^{(1,0)}_{n}\nonumber \\
& + \oJ_n^{(1,0)} \Gamma_{2n+1,k} \Gamma_{2n,k}(\oJ_{n-1}^{(1,0)})^\top\widehat{\mP}^{(1,0)}_{n-1},\label{Small-P1-TTR}\\
x_k\,\widehat{\mP}^{(0,1)}_{n} = & \oL_{n,k} \widehat{\mP}^{(0,1)}_{n+1} 
+ \oJ_n^{(0,1)}[ \oL_{2n+1,k} \Gamma_{2n+2,k} + \Gamma_{2n+1,k} \oL_{2n,k}](\oJ_{n}^{(0,1)})^\top
\widehat{\mP}^{(0,1)}_{n}\nonumber \\
& + \oJ_n^{(0,1)}\Gamma_{2n+1,1} \Gamma_{2n,1}(\oJ_{n-1}^{(0,1)})^\top\widehat{\mP}^{(0,1)}_{n-1},\label{Small-R1-TTR}
\end{align}
for $k=1,2$ and $x_1=x, x_2 = y$.
\end{lemma}
Now, the B\"acklund-type relations contained in Theorem~\ref{Backlund} are proven identifying~coefficients.

As we have shown in Theorems \ref{direct} and \ref{converse}, the \emph{small} polynomial systems $\{\widehat{\mP}_n^{(i,j)}\}_{n\geqslant0}$, for $i+j \geqslant1$ are Christoffel modifications of the first family $\{\widehat{\mP}_n^{(0,0)}\}_{n\geqslant0}$. Then, by Lemma \ref{shortRel}, there exist short relations between that families. Lemma~\ref{ShortRel_BigPoly} also allows us to deduce short relations for the \emph{small} polynomial systems, multiplying by the adequate $\oJ$-matrix, and using Lemma~\ref{J-L}. Next result gives the coefficients in terms of the matrix coefficients of the three term relations of $\{\mathbb{S}_n\}_{n\geqslant0}$.

\begin{corollary}
The families of \emph{small} MOPS are related by
\begin{align*}
\widehat{\mP}^{(0,0)}_{n}(x, y) = & \widehat{\mP}^{(2-k,k-1)}_{n}(x, y) 
+ \widehat{\Gamma}_{n,k}^{(0,0)}\,\widehat{\mP}^{(2-k,k-1)}_{n-1}(x, y), \\
x_k\, \widehat{\mP}^{(1,1)}_{n-1}(x, y) = & \oL_{n-1,k}\,\widehat{\mP}^{(k-1,2-k)}_{n}(x, y) 
+ \widehat{\Gamma}_{n,2}^{(0,1)}\,\widehat{\mP}^{(k-1,2-k)}_{n-1}(x, y), \\
\widehat{\mP}^{(k-1,2-k)}_{n}(x, y) = & \widehat{\mP}^{(1,1)}_{n}(x, y) 
+ \widehat{\Gamma}_{n,k}^{(1,1)}\, \widehat{\mP}^{(1,1)}_{n-1}(x, y), \\
x_k\,\widehat{\mP}^{(2-k,k-1)}_{n}(x, y) = & \oL_{n,k}\, \widehat{\mP}^{(0,0)}_{n+1}(x, y) 
+ \widehat{\Gamma}_{n,k}^{(1,0)}\,\widehat{\mP}^{(0,0)}_{n}(x, y),
\end{align*}
where 
\begin{align*}
& \widehat{\Gamma}_{n,k}^{(0,0)} = \oJ_{n}^{(0,0)}\,\Gamma_{2n,k}\,(\oJ_{n-1}^{(1,0)})^\top,
& \widehat{\Gamma}_{n,k}^{(0,1)} = \oJ_{n-1}^{(1,1)}\,\Gamma_{2n,k}\,(\oJ_{n-1}^{(0,1)})^\top, \\ 
& \widehat{\Gamma}_{n,k}^{(1,1)} = \oJ_{n}^{(0,1)}\,\Gamma_{2n+1,k}\,(\oJ_{n-1}^{(1,1)})^\top, 
& \widehat{\Gamma}_{n,k}^{(1,0)} = \oJ_{n}^{(1,0)}\,\Gamma_{2n+1,k}\,(\oJ_{n}^{(0,0)})^\top. 
\end{align*}
\end{corollary}

These matrices $\widehat{\Gamma}$'s enable us to reinterpret the block Jacobi matrix associated with the polynomials sequences~$\widehat{\mP}$'s in terms of a $\pmb{\mathsf{L}} \pmb{\mathsf{U}}$ or $\pmb{\mathsf{U}} \pmb{\mathsf{L}}$ representation. In fact,
for $k=1,2$, we define the block matrices
\begin{align*}
\pmb{\mathsf{L}}^{0}_k
=
\begin{bmatrix}
 \operatorname I \\
 \widehat{\Gamma}_{1,k}^{(0,0)} & \operatorname I \\
 & \widehat{\Gamma}_{2,k}^{(0,0)} & \operatorname I \\
 & & \ddots & \ddots
\end{bmatrix},
 &&
\pmb{\mathsf{L}}^{1}_k
=
\begin{bmatrix}
 \operatorname I \\
 \widehat{\Gamma}_{1,k}^{(1,1)} & \operatorname I \\
 & \widehat{\Gamma}_{2,k}^{(1,1)} & \operatorname I \\
 & & \ddots & \ddots
\end{bmatrix},\\
\pmb{\mathsf{U}}^{0}_k
=
\begin{bmatrix}
 \widehat{\Gamma}_{1,k}^{(0,1)} & \operatorname L_{0,k} \\
 & \widehat{\Gamma}_{2,k}^{(0,1)} & \operatorname L_{1,k} \\
 & & \ddots & \ddots
\end{bmatrix},
 &&
\pmb{\mathsf{U}}^{1}_k
=
\begin{bmatrix}
 \widehat{\Gamma}_{1,k}^{(1,0)} & \operatorname L_{0,k} \\
 & \widehat{\Gamma}_{2,k}^{(1,0)} & \operatorname L_{1,k} \\
 & & \ddots & \ddots
\end{bmatrix},
\end{align*}
we recover the recurrence relations~\eqref{Small-P0-TTR}, \eqref{Small-R0-TTR}, \eqref{Small-P1-TTR}, \eqref{Small-R1-TTR},  respectively
\begin{align*}
x_k \, \pmb{\mathcal P}^{(0,0)} & = x_k \, \pmb{\mathsf{L}}^{0}_k \, \pmb{\mathcal P}^{(2-k,k-1)}
 = \pmb{\mathsf{L}}^{0}_k \, \pmb{\mathsf{U}}^{1}_k \, \pmb{\mathcal P}^{(0,0)} , \\
 x_k \, \pmb{\mathcal P}^{(1,1)} & = \pmb{\mathsf{U}}^{0}_k \, \pmb{\mathcal P}^{(k-1,2-k)} 
 = \pmb{\mathsf{U}}^{0}_k \, \pmb{\mathsf{L}}^{1}_k \, \pmb{\mathcal P}^{(1,1)} , \\
x_k \, \pmb{\mathcal P}^{(2-k, k-1)} & = \pmb{\mathsf{U}}^{1}_k \, \pmb{\mathcal P}^{(0,0)}
 = \pmb{\mathsf{U}}^{1}_k \, \pmb{\mathsf{L}}^{0}_k \, \pmb{\mathcal P}^{(2-k, k-1)} ,  \\
x_k \, \pmb{\mathcal P}^{(k-1,2-k)} & = x_k \,\pmb{\mathsf{L}}^{1}_k \,  \pmb{\mathcal P}^{(1,1)}
 = \pmb{\mathsf{L}}^{1}_k \, \pmb{\mathsf{U}}^{0}_k \, \pmb{\mathcal P}^{(k-1,2-k)} , 
\end{align*}
denoting $x_1=x$, $x_2=y$, and, for $i,j=0,1$, the column vector $\pmb{\mathcal P}^{(i,j)}$ is defined as
\begin{align*}
\pmb{\mathcal P}^{(i,j)} =
\begin{bmatrix}
(\widehat{\mathbb P}^{(i,j)}_0)^\top &
(\widehat{\mathbb P}^{(i,j)}_1)^\top &
\cdots
\end{bmatrix}^\top.
\end{align*}

\section{A case study}

Moreover, if a weight function can be represented as $W(x,y) = \widetilde{W}(x^2, y^2)$, then 
it is $x\,y$-symmetric.

Finally, we totally describe the connection between bivariate polynomials orthogonal with respect to a $x\,y$-symmetric weight function defined on the unit ball of $\mathbb{R}^2$, defined by 
\begin{align*}
\mathbf{B}^2 = \{(x,y)\in \mathbb{R}^2: x^2 + y^2 \leqslant 1\},
\end{align*}
and bivariate orthogonal polynomials defined on the simplex
\begin{align*}
\mathbf{T}^2 = \{(x,y)\in\mathbb{R}^2: x, y \geqslant 0, x+y\leqslant 1\},
\end{align*}
completing the discussion started by Y. Xu in \cite{Xu98, Xu01} for the even ball polynomials in each of its variables. 

Following \cite[section 4]{Xu01}, let $W^\mathbf{B}(x,y)= W(x^2,y^2)$ be a weight function defined on the unit ball on $\mathbb{R}^2$, and let 
\begin{align}\label{weigh_T}
W^\mathbf{T}(u,v) = \frac{1}{\sqrt
{u\,v}} \, W(u,v), && (u, v)\in \mathbf{T}^2.
\end{align}
Observe that $W^\mathbf{B}(x,y)$ is a $x\,y$-symmetric weight function defined on $\mathbf{B}^2$.

For $n\geqslant0$, and $0\leqslant k \leqslant n$, let $S_{2n,2k}(x, y)$ be an orthogonal polynomial associated to the weight function $W^\mathbf{B}$ of even degree in each of variables. Then Y. Xu proved that it can be written in terms of orthogonal polynomials on the simplex as
\begin{align*}
S_{2n,2k}(x,y) = P_{n,k} (x^2,y^2), 
\end{align*}
where $P_{n,k} (x,y)$ is an orthogonal polynomial of total degree $n$ associated to $W^\mathbf{T}$. 

We can answer the question that what is about the leftover polynomials,~\emph{i.e.}, we can give explicitly the shape of the polynomials orthogonal with respect to $W^\mathbf{B}(x,y)$. Following our results, these polynomials are related to new families of bivariate orthogonal polynomials, resulting from a Christoffel modification that we will explicitly identify. 

Let $\{\mathbb{S}_n\}_{n\geqslant0}$ be the monic orthogonal polynomial system associated with the $x\,y$-symmetric weight function $W^\mathbf{B}(x,y)$, satisfying \eqref{IP_S}. 

If the explicit expression of every monic vector polynomial is given by
\begin{align*}
\mS_n(x,y)
= \begin{bmatrix}
S_{n,0}(x,y) & S_{n,1}(x,y) & S_{n,2}(x,y) & \cdots & S_{n,n}(x,y)
\end{bmatrix}^\top,
\end{align*}
then every polynomial $S_{n,k}(x,y)$, for $0\leqslant k \leqslant n$, is $x\,y$-\emph{symmetric} by Lemma \ref{xy-sim}. As we have proved, the vector of polynomials $\mS_{n}(x,y)$ can be separated in a \emph{zip} way,
\emph{cf.}~\eqref{separation}, attending to the parity of the powers of $x$ and $y$, in its entries.

We deduce four families: $\{S_{2n,2k}(x,y): 0\leqslant k \leqslant n\}_{n\geqslant0}$, 
$\{S_{2n,2k+1}(x,y): 0 \leqslant k \leqslant n-1\}_{n\geqslant0}$, $\{S_{2n+1,2k}(x,y): 0\leqslant k \leqslant n\}_{n\geqslant0}$, and $\{S_{2n+1,2k+1}(x,y): 0\leqslant k \leqslant n\}_{n\geqslant0}$. Only the first family was identified in \cite{Xu98, Xu01} under the transformation $(x^2,y^2) \mapsto (x,y)$ as a family of polynomials orthogonal on $\mathbf{T}^2$ with respect to the weight function \eqref{weigh_T}. We observe that the second family has the common factor $x\,y$, the third family has $x$ as common factor, and the fourth family has common factor the second variable $y$.

Working as in Section~\ref{symm_2}, we separate the symmetric monic orthogonal polynomial vectors as
it was shown in Lemma~\ref{representation}
\begin{align*} 
\mS_{2n}(x,y) & = \mP_n^{(0,0)}(x^2, y^2) + x\,y\,\mP_{n-1}^{(1,1)}(x^2, y^2),\\
\mS_{2n+1}(x,y) & = x\, \mP_n^{(1,0)}(x^2, y^2) + y\,\mP_{n}^{(0,1)}(x^2, y^2) .
\end{align*}
After deleting all zeros in above vectors of polynomials and substituting the variables $(x^2,\,y^2)$ by $(x,\,y)$, we proved, in Theorem~\ref{direct} that

\medskip

\noindent
$\{\widehat{\mP}_n^{(0,0)}\}_{n\geqslant0} = \{\oJ_{n}^{(0,0)}\,\mP_n^{(0,0)}\}_{n\geqslant0}$
is a MOPS associated with the weight~function 
\begin{align*}
W^{(0,0)}(x,y) = \dfrac{1}{\sqrt{x\,y}}\,W^{\mathbf{B}}(\sqrt{x},\sqrt{y}),
\end{align*}

\noindent
$\{\widehat{\mP}_n^{(2-k,k-1)}\}_{n\geqslant0} = \{\oJ_{n}^{(2-k,k-1)}\,\mP_n^{(2-k,k-1)}\}_{n\geqslant0}$, for $k=1,2$,
are MOPS associated with the Chris\-tof\-fel modification 
\begin{align*}
W^{(2-k,k-1)}(x,y) = x_k \, W^{(0,0)}(x,y), 
\end{align*}

\noindent
$\{\widehat{\mP}_n^{(1,1)}\}_{n\geqslant0} = \{\oJ_{n}^{(1,1)}\,\mP_n^{(1,1)}\}_{n\geqslant0}$
is a MOPS associated with the Christoffel modification
\begin{align*}
W^{(1,1)}(x,y) = x\,y\, W^{(0,0)}(x,y),
\end{align*}
for all $ (x,y) \in \mathbf{T}^2$.

Therefore, we have described the complete relation between orthogonal polynomials on the ball with orthogonal polynomials on the simplex.

\section*{Acknowledgements}

AB acknowledges Centro de Matem\'atica da Universidade de Coimbra (CMUC) -- UID/MAT/ 00324/2020, funded by the Portuguese Government through FCT/MEC and co-funded by the European Regional Development Fund through the Partnership Agreement PT2020.

AFM acknowledges CIDMA Center for Research and Development in Mathematics and Applications (University of Aveiro) and the Portuguese Foundation for Science and Technology (FCT)
within project UID/MAT/04106/2020.

TEP thanks FEDER/Junta de Andaluc\'ia under the research project A-FQM-246-UGR20; MCIN/AEI 10.13039/501100011033 and FEDER funds by PGC2018-094932-B-I00; and IMAG-Mar\'ia de Maeztu grant CEX2020-001105-M.


\end{document}